\documentclass[11pt,reqno]{amsart}
\usepackage[utf8]{inputenc}
\usepackage{enumitem}
\usepackage{amssymb}
\usepackage{mathtools}
\usepackage{pdfsync}
\usepackage[english]{babel}
\usepackage[pdftex,pagebackref,colorlinks=true,urlcolor=blue,linkcolor=blue,citecolor=blue]{hyperref}
\usepackage[capitalize]{cleveref}
\usepackage{fullpage}
\usepackage{color}
\usepackage{amsmath}
\usepackage{amsfonts}
\usepackage{mathrsfs}
\usepackage{t1enc, graphicx}
\usepackage{verbatim}
\usepackage{bbm}
\usepackage[colorinlistoftodos]{todonotes}
\usepackage{enumitem}
\usepackage{bm}
\usepackage{tikz-cd}

\newcommand{\Z}{\mathbb{Z}}
\newcommand{\N}{\mathbb{N}}

\newcommand{\RP}{\bold{RP}}

\newcommand{\bQ}{\bm{Q}}

\newtheorem{theorem}{Theorem}[section]
\newtheorem{proposition}[theorem]{Proposition}

\newtheorem{problem}[theorem]{Problem}
\newtheorem{lemma}[theorem]{Lemma}

\newtheorem{corollary}[theorem]{Corollary}
\theoremstyle{definition}
\newtheorem{definition}[theorem]{Definition}
\newtheorem*{definition*}{Definition}

\newtheorem{remark}[theorem]{Remark}
\theoremstyle{remark}

\title{Joint transitivity for linear iterates}

\date{}

\begin{document}

\author{Sebasti{\'a}n Donoso, Andreas Koutsogiannis and Wenbo~Sun}

\address[Sebasti{\'a}n Donoso]{Departamento de Ingenier\'{\i}a Matem\'atica and Centro de Modelamiento Matem{\'a}tico, Universidad de Chile \& IRL 2807 - CNRS, Beauchef 851, Santiago, Chile} \email{sdonoso@dim.uchile.cl}

\address[Andreas Koutsogiannis]{
Department of Mathematics, Aristotle University of Thessaloniki, Thessaloniki, 54124, Greece}

\email{akoutsogiannis@math.auth.gr}

\address[Wenbo Sun]{Department of Mathematics, Virginia Tech, 225 Stanger Street, Blacksburg, VA, 24061, USA}
\email{swenbo@vt.edu}

\thanks{The first author was partially funded by Centro de Modelamiento Matemático (CMM) FB210005, BASAL funds for centers of excellence from ANID-Chile and ANID/Fondecyt/1241346. The second author was partially supported by the ``Excellence in Research'' program of the Special Account for Research Funds AUTH (Code 10316). The third author was partially supported by the NSF Grant DMS-2247331.}



\maketitle
\begin{abstract} 
\noindent We establish sufficient and necessary conditions for the joint transitivity of linear iterates in a minimal topological dynamical system with commuting transformations. This result provides the first topological analogue of the classical Berend and Bergelson joint ergodicity criterion in measure-preserving systems. 
\end{abstract}

\section{Introduction - Main result}

\subsection{The joint ergodicity problem} Let $(X,\mathcal{B},\mu)$ be a standard probability space equipped with an invertible measure-preserving transformation $T:X\to X$ (that is, $\mu(TA)=\mu(A)$ for every $A\in \mathcal{B}$). We say that the quadruple $(X,\mathcal{B},\mu,T)$ is \emph{a measure-preserving system}. In particular, the latter is called \emph{ergodic} or \emph{weakly mixing} if the transformation $T$ is ergodic (i.e., every $T$-invariant set $A\in \mathcal{B}$ satisfies $\mu(A)\in \{0,1\}$) or weakly mixing (i.e., the transformation $T\times T,$ acting on the Cartesian square $X^2:=X\times X$, is ergodic) respectively. 

Given a weakly mixing measure-preserving system $(X,\mathcal{B},\mu,T)$ and distinct non-zero integers $a_1,\ldots,a_d,$ we have the following independence property of the sequences $(T^{a_i n})_n,$ $1\leq i\leq d.$\footnote{Throughout this paper, whenever a sequence is written as $(a(n))_{n}$ without specifying the range of $n$, it is understood as a $\mathbb{Z}$-sequence $(a(n))_{n\in\mathbb{Z}}$.}

\begin{theorem}[\cite{Furstenberg_ergodic_szemeredi:1977}]\label{T:F1}
Let $(X,\mathcal{B},\mu,T)$ be a weakly mixing measure-preserving system. Then, for any $d\in \N, $ any distinct non-zero integers $a_1,\ldots,a_d,$ and any $f_1,\ldots,f_d\in L^\infty(\mu)$ we have 
\begin{equation}\label{E:F1}\lim_{N\to\infty}\frac{1}{N}\sum_{n=1}^N T^{a_1 n} f_1\cdot\ldots\cdot T^{a_d n}f_d=\int_X f_1\;d\mu\cdot\ldots\cdot\int_X f_d\;d\mu,\end{equation}
where the convergence takes place in $L^2(\mu).$
\end{theorem}

This result, in particular its recurrence reformulation on $\mathcal{B}$-measurable sets of positive measure, is a crucial ingredient in Furstenberg's approach in proving Szemer\'edi's theorem (that is, every subset of natural numbers of positive upper density contains arbitrarily long arithmetic progressions) by recasting it as a recurrence problem. 

Later, in \cite{Bergelson_WM_PET:1987}, Bergelson extended Theorem~\ref{T:F1} to essentially distinct integer polynomial iterates. The convergence of general multiple ergodic averages for various classes of iterates to the right-hand side of \eqref{E:F1}, also known as the ``expected limit,'' developed to be a topic on its own (e.g., see \cite{Bergelson_WM_PET:1987,Bergelson_Moreira_Richter_mult_averages_conv_recurrence_comb_app:2024,Chu_Frantzikinakis_Host_ergodic_averages_distinct_degree:2011,Donoso_Ferre_Koutsogiannis_Sun_multicorr_joint_erg:2024,Donoso_Koutsogiannis_Sun_pointwise_sublinear:2020,Donoso_Koutsogiannis_Sun_seminorms_polynomials_joint_ergodicity:2022,Frantzikinakis_mult_recurrence_Hardy_poly_growth:2010,Frantzikinakis_multidim_Szemeredi_Hardy:2015,Frantzikinakis_joint_ergodicity_sequences:2023,Frantzikinakis_Kra_averages_product_integrals:2005,Karageorgos_Koutsogiannis_integer_indep_poly_averages_and_primes:2019,Koutsogiannis_integer_poly_correlation:2018,Koutsogiannis_multiple_averages_variable_poly:2022,Koutsogiannis_Sun_total_joint_ergodicity:2023,Tsinas_joint_erg_Hardy:2023} for various results on polynomial and  Hardy field, of polynomial growth, functions); the one of \emph{joint ergodicity}.

\begin{definition} Let $(X,\mathcal{B},\mu,T_{1},\dots,T_{d})$ be a measure-preserving system with commuting and invertible transformations\footnote{ Naturally, by this we mean that $(X,\mathcal{B},\mu)$ is a standard probability space and $T_1,\ldots, T_d:X\to X$ are invertible measure-preserving transformations with $T_i T_j=T_j T_i.$} and $(a_1(n))_n,\ldots,(a_d(n))_n$ be integer-valued sequences.  We say that $(T_1^{a_1(n)})_n,\ldots,(T_d^{a_d(n)})_n$ are \emph{jointly ergodic (for $\mu$)} if for any $f_1,\ldots, f_d\in L^\infty(\mu)$ we have 
\begin{equation}\label{E:JE2}
\lim_{N\to\infty}\frac{1}{N}\sum_{n=1}^N  T_1^{a_1(n)}f_1\cdot\ldots\cdot T_d^{a_d(n)} f_d=\int_X f_1\;d\mu\cdot\ldots\cdot\int_X f_d\;d\mu,\end{equation} where the convergence takes place in $L^2(\mu)$ (i.e., the $L^2(\mu)$-limit of the left-hand side of \eqref{E:JE2} exists and it takes the value of the right-hand side).
\end{definition}

The first characterization of joint ergodicity is due to Berend and Bergelson.

\begin{theorem}[\cite{Berend_Bergelson_joint_ergodicity:1984}]\label{T:BB}
 Let $(X,\mathcal{B},\mu,T_{1},\dots,T_{d})$ be a measure-preserving system with commuting and invertible transformations. 
 Then $(T_{1}^n)_n,\dots,(T_{d}^n)_n$ are jointly ergodic (for $\mu$) if and only if both of the following conditions are satisfied:
	\begin{itemize}
		\item[$(i)$] $T_{i}T^{-1}_{j}$ is ergodic (for $\mu$) for all $1\leq i,j\leq d,$ $i\neq j$; and
		\item[$(ii)$] $T_{1}\times \dots\times T_{d}$ is ergodic (for $\mu^{\otimes d}$).
	\end{itemize}
\end{theorem}

There is a plethora of analogous joint ergodicity characterizations for generalized linear functions \cite{Bergelson_Leibman_Son_joint_erg_generalized_linear:2016}, polynomial functions \cite{Donoso_Ferre_Koutsogiannis_Sun_multicorr_joint_erg:2024, Donoso_Koutsogiannis_Sun_seminorms_polynomials_joint_ergodicity:2022, Frantzikinakis_Kuca_joint_erg_commuting:2022, Frantzikinakis_Kuca_seminorm_control_commuting_poly:2022}, and Hardy field functions \cite{Donoso_Koutsogiannis_Kuca_Tsinas_Sun_multiple_Hardy, Donoso_Koutsogiannis_Sun_joint_erg_poly_growth:2023}. 

The objective of this article is to prove the corresponding to Theorem~\ref{T:BB} result in the topological setting.

\subsection{The joint transitivity problem}

A \emph{$\Z^{k}$-system} is a tuple $(X,S_1,\ldots,S_k)$ where $X$ is a compact metric space and $S_{1},\dots,S_{k}\colon X\to X$ are homeomorphisms with $S_i S_j= S_j S_i$. We let $\langle S_1,\ldots,S_k\rangle$ denote the group generated by $S_1,\ldots,S_k$. 

In order to state the result corresponding to Theorem~\ref{T:BB} in the topological setting, we start with the notion of {\em joint transitivity}.

\begin{definition}
Let $(X,S_1,\ldots,S_k)$ be a $\Z^k$-system, $T_1,\ldots,T_d\in \langle S_1,\ldots,S_k\rangle$ and $(a_1(n))_n,\ldots,$ $(a_d(n))_n$ be integer-valued sequences. We say that $(T_1^{a_1(n)})_n,\ldots,(T_d^{a_d(n)})_n$ are \emph{jointly transitive} if there is a $G_{\delta}$-dense subset $X_0\subseteq X$ such that for all $x\in X_0$, the set
\[ \{ (T_1^{a_1(n)}x,\ldots,T_d^{a_d(n)}x) : n\in \Z\} \]
is dense in $X^d$.\footnote{ In \cite{Huang_Shao_Ye_top_correspondence_multiple_averages:2019}, the authors call the family $(T^{n}_{1},\dots,T^{n}_{d})$ to be \emph{$\Delta$-transitive}.  We use the term ``joint transitivity'' to emphasize the direct parallelism of Theorem~\ref{T:BB} to \cref{mainbb} in the measure theoretic setting on ``joint ergodicity.''}
\end{definition}

We call a $\Z^k$-system $(X,S_1,\ldots,S_k)$ {\em minimal} if, for any point $x\in X$, its orbit $\{ S_1^{m_1}\cdot\ldots\cdot S_k^{m_k} x: (m_1,\ldots,m_k)\in \Z^k\}$ is dense in $X$.  A system is {\em transitive} if there exists a point $x\in X$ whose orbit is dense; we say that any such point $x$ is a \emph{transitive} point (of $(X,S_1,\ldots,S_k)$).  The pioneer work of Glasner \cite{Glasner_top_erg_decomposition:1994}, established joint transitivity for sequences given by iterates of powers of a transformation in a minimal and (topologically) weakly mixing $\mathbb{Z}$-system $(X,T)$ (meaning that the product system $(X\times X, T\times T)$ is transitive).

\begin{theorem}[\cite{Glasner_top_erg_decomposition:1994}]\label{ttgl}
  Let $(X,T)$ be a minimal weakly mixing $\mathbb{Z}$-system. Then, for any $d\in \N$, and distinct nonzero integers $a_1,\ldots,a_d$, the sequences $(T^{a_1 n})_n,\ldots,(T^{a_d n})_n$ are jointly transitive.  
\end{theorem}

 Theorem \ref{ttgl} can be thought as the topological analogue of \cref{T:F1}.
This result was extended by Huang, Shao and Ye in 
\cite{Huang_Shao_Ye_top_correspondence_multiple_averages:2019}, who obtained topogically joint ergodicity results under weakly mixing assumptions of several transformations but were able to deal with polynomial expressions and nilpotent group actions. Some follow-up works on this line are given in \cite{Cao_Shao_top_mild_mixing_poly:2022,Zhang_Zhao_topological_mult_rec_WM_GP:2021}.

We are now ready to state our main result, which can be regarded as the topological version of Theorem~\ref{T:BB}.

 \begin{theorem}\label{mainbb}
Let $(X,S_1,\ldots,S_k)$ be a minimal system and $T_1,\ldots,T_d \in \langle S_1,\ldots,S_k\rangle$. Then $(T_1^n)_n,\ldots,(T_d^n)_n$ are jointly transitive if and only if both of the following conditions are satisfied:
\begin{enumerate}[label=(\roman*)]
    \item[$(i)$] $(X,T_i^{-1}T_j)$ is transitive for all $1\leq i, j\leq d,\; i\neq j$; and
    
    \item[$(ii)$] $(X^d,T_1\times \cdots \times T_d)$ is transitive. 
\end{enumerate}
\end{theorem}

\begin{remark} As in the measurable case with Theorem~\ref{T:BB}, Theorem~\ref{mainbb} provides a characterization for all linear iterates. 

Indeed, assuming that the iterate of the $T_i$ is $a_i n+b_i,$ $a_i\in \mathbb{Z}\setminus\{0\}, b_i\in \mathbb{Z},$ $1\leq i\leq d,$ noting that the shifts by the $b_i$'s do not affect the denseness of the orbits, we can use the Theorem~\ref{mainbb} for the functions $T^{a_i}$ (which still belong to $\langle S_1,\ldots,S_k\rangle$). 
\end{remark}

It is important to note that the problem in the topological setting differs significantly from the one in the measurable setting as the dense $G_{\delta}$ subset of $X$ might have zero measure. For example, a minimal weakly mixing system $(X,T)$ may exhibit a discrete spectrum with respect to some invariant measure. In such a system, $(T^n)_n,\ldots(T^{dn})_n$ are jointly transitive but not jointly ergodic. In fact, any ergodic measure preserving system is measurably isomorphic to a minimal and uniquely ergodic (strongly) mixing system (see \cite{Lehrer_mixing_unique_erg_models:1987}). 

We also want to emphasize that \cref{mainbb} fails without the minimality assumption. Indeed, in \cite{Moothathu_diagonal_points:2010}, Moothathu showed that there exists a non-minimal, strongly mixing shift $(X,\sigma)$ such that, for every point $x\in X$, the set $\{(\sigma^n x, \sigma^{2n}x):n\in \Z\}$ fails to be dense in $X^2$. 
Because $(X,\sigma)$ is strongly mixing, the $\Z^2$-system $(X,\sigma,\sigma^2)$ satisfies conditions $(i)$ and $(ii)$ of \cref{mainbb}, but the sequences $(\sigma^n)_n,(\sigma^{2n})_n$ are not jointly transitive. It should be noted that there are no commuting transformations $S_1,\ldots,S_k$ that generate a minimal action and such that $\sigma\in \langle S_1,\ldots,S_k\rangle$. One reason for this is, of course, \cref{mainbb}, but this can also be seen directly in Moothathus's example, using the fact that the set of $\sigma$-periodic points of a given period have to be preserved by $\langle S_1,\ldots,S_k\rangle$, which prevents minimality.

\subsection{Structure of the paper}
 In Section~\ref{S:2} we recall some notions from the theory of dynamical systems. In particular, we provide equivalent statements to joint transitivity (Lemma \ref{le:eqtr}), and we list properties of dynamical cubes and regional proximal relations.

In Section~\ref{S:3} we first characterize (in Theorem~\ref{lemma:product_RP}) regional proximal relations for product transformations and, finally, we prove Theorem~\ref{mainbb} by an inductive argument.

\section{Background material and useful facts}\label{S:2}

\noindent\textbf{Definitions and conventions.}
For any $\mathbb{Z}^{k}$-system $(X,S_{1},\dots,S_{k})$ and $m=(m_{1},\dots,m_{k})\in\mathbb{Z}^{k}$, we write $S_{m}\coloneqq S_{1}^{m_{1}}\cdot\ldots\cdot S_{k}^{m_{k}}$. So, we may write a $\mathbb{Z}^{k}$-system as $(X,(S_{m})_{m\in\mathbb{Z}^{k}})$ whenever we do not need to stress the generators.
With this convention, a $\mathbb{Z}^k$-system is minimal if for any point $x\in X$, its orbit $\{ S_{m} x: m\in \mathbb{Z}^k\}$ is dense in $X$.
We adopt a similar notation for subgroups of $\mathbb{Z}^k$. If $G\subseteq \mathbb{Z}^k$ is a subgroup, then $(X,(S_{m})_{m\in G})$ denotes the system given by the subaction of $G$.

We will use $\rho$ to denote the metric on $X$, and slightly abusing notation, we will use $\rho$ to denote the metric on the product space $X^d$ as well, where $\rho((x_1,\ldots,x_d),(x_1',\ldots,x_d'))=\sup_{1\leq i\leq d} \rho(x_i,x_i')$.

Let $(X,S_{1},\dots,S_{k})$ and $(Y,R_{1},\dots,R_{k})$ be two $\Z^{k}$-systems. We say that $Y$ is a \emph{factor} of $X$ (or that $X$ is an \emph{extension} of $Y$) if there exists a continuous and onto map $\pi\colon X\to Y$ (called the \emph{factor map} from $X$ to $Y$) such that $\pi\circ S_{i}=R_{i}\circ \pi$ for all $1\leq i\leq k$. When $\pi$ is a homeomorphism, we say that the systems are {\em topologically conjugate}.

There is a one-to-one correspondence between the factors and the closed invariant equivalence relations on $X$. Indeed, we can associate a factor map $\pi\colon X\to Y$ with the  relation $R_{\pi}=\{(x,y): \pi(x)=\pi(y)\}$, and, conversely, given a closed invariant equivalence relation $R$, we can associate it with a factor map $X\to X/R$ being the quotient map. 
A factor map $\pi\colon X\to Y$ is {\em almost one-to-one} if there exists a $G_{\delta}$-dense subset $\Omega$ of $X$ such that for any $x\in \Omega$, $\pi^{-1}(\pi(x))=\{x\}$. A factor map $\pi\colon X\to Y$ is {\em open} if $\pi(A)\subseteq Y$ is open whenever $A\subseteq X$ is open. Note that the latter implies that for any $x,x'\in R_{\pi}$, and $\epsilon>0$, there exists $\delta>0$ such that if $\rho(x,y)<\delta$, then there exists $y'\in X$ with $\rho(x',y')<\epsilon$ and $(x',y')\in R_{\pi}$.   

A system $(X,S_1,\ldots,S_k)$ is {\em equicontinuous} if the family of functions generated by $S_1,\ldots,S_k$ is equicontinuous. Any minimal equicontinuous $\Z^k$-system is topologically conjugate to a rotation on a compact abelian group (see \cite[Chapter 2]{Auslander_minimal_flows_and_extensions:1988}).
The {\em maximal equicontinuous factor} of a $\Z^k$-system $(X,S_1,\ldots,S_k)$, is the largest equicontinuous factor of it. That is, any equicontinuous factor of $(X,S_1,\ldots,S_k)$ is also a factor of the maximal one.  The maximal equicontinuous factor of a minimal $\Z^k$-system is induced by the regionally proximal relation $\RP_{\Z^k,\Z^k}(X)$ (see \cref{subsec:DC_RP}).

Let $(X,S_1,\ldots,S_k)$ be a $\Z^k$-system. A pair $(x,y)$ is {\em proximal} if there exists a sequence $(n_i)_i$ in $\Z^k$ such that $\rho(S_{n_i}x,S_{n_i}y)\to 0$ as $i$ goes to infinity. The set of all proximal pairs is denoted by $P(X)$. It is well known that $P(X)\subseteq \RP_{\Z^k,\Z^k}(X)$ (see, for instance, \cite{Shao_Ye_regionally_prox_orderd:2012}). A factor map $\pi\colon X\to Y$ is {\em proximal} if $R_{\pi}\subseteq P(X)$. Any almost one-to-one factor map between minimal systems is proximal (see \cite[Chapter 11]{Auslander_minimal_flows_and_extensions:1988}).

\subsection{Equivalent definitions for joint transitivity}
The following lemma provides a couple of equivalent definitions for joint transitivity. We will use this lemma implicitly throughout this paper. Its proof is a direct generalization of \cite[Lemma 2.8]{Qiu_poly_orbits_tot_minimal:2023} (see also \cite[Lemma 2.4]{Huang_Shao_Ye_top_correspondence_multiple_averages:2019} and \cite{Moothathu_diagonal_points:2010}).

\begin{lemma}\label{le:eqtr}
Let $(X,S_{1},\dots,S_{k})$ be a minimal $\Z^k$-system and $(a_{1}(n))_n,\ldots,(a_{d}(n))_n$ be sequences with values in $\mathbb{Z}^k$. The following are equivalent: 
    \begin{enumerate}[label=(\roman*)]
        \item[$(i)$] There exists a dense $G_{\delta}$ subset $\Omega$ of $X$ such that the set
        $$\{(S_{a_{1}(n)}x,\dots,S_{a_{d}(n)}x)\colon n\in\mathbb{Z}\}$$
        is dense in $X^{d}$ for every $x\in\Omega$.
         \item[$(ii)$] There exists some $x\in X$ such that the set
        $$\{(S_{a_{1}(n)}x,\dots,S_{a_{d}(n)}x)\colon n\in\mathbb{Z}\}$$
        is dense in $X^{d}$.
        \item[$(iii)$] For every non-empty open subsets $U,V_{1},\dots,V_{d}$ of $X$, there is some $n\in\mathbb{Z}$ such that 
        $$U \cap S_{-a_{1}(n)}V_{1}\cap \dots\cap S_{-a_{d}(n)}V_{d}\neq\emptyset.$$
    \end{enumerate}
\end{lemma}

\begin{proof}
 The implication $(i)$ $\Rightarrow$ $(ii)$ is obvious. We next prove that $(ii)$ implies $(iii)$. To this end, let $x\in X$ be such that the set
        $$\{(S_{a_{1}(n)}x,\dots,S_{a_{d}(n)}x)\colon n\in\mathbb{Z}\}$$
        is dense in $X^{d}$. Then, for any $m\in\Z^{k}$, the set 
         $$X(x,m):=\{(S_{a_{1}(n)+m}x,\dots,S_{a_{d}(n)+m}x)\colon n\in\mathbb{Z}\}$$
        is also dense in $X^{d}$.  Now, for any non-empty open subsets $U,V_{1},\dots,V_{d}$ of $X$, by the minimality of $(X,S_{1},\dots,S_{k})$, we may find some $m\in\Z^{k}$ such that $S_{m}x\in U,$ and since $X(x,m)$
        is dense in $X^{d}$, there exists $n\in\Z$ such that $S_{a_{i}(n)+m}x\in V_{i}$ for all $1\leq i\leq d$. Therefore,
        the set
        \[U \cap S_{-a_{1}(n)}V_{1}\cap \dots\cap S_{-a_{d}(n)}V_{d}\]
        contains the point $T_{m}x,$ hence it is non-empty.

        It remains to show that $(iii)$ implies $(i)$. Let $\mathcal{F}$ be a countable basis of the topology of $X$ and define
$$\Omega\coloneqq \bigcap_{V_{1},\dots,V_{d}\in\mathcal{F}}\bigcup_{n\in\Z}S_{-a_{1}(n)}V_{1}\cap \dots\cap S_{-a_{d}(n)}V_{d}.$$
It follows from $(iii)$ that $\Omega$ is a dense $G_{\delta}$ set. Moreover, the set
\[\{(S_{a_{1}(n)}x,\dots,S_{a_{d}(n)}x)\colon n\in\mathbb{Z}\}\] is dense in $X^{d}$ for every $x\in\Omega$. 
\end{proof}

\subsection{Dynamical cubes and regionally proximal relations} \label{subsec:DC_RP}
Let $(X,S_1,\ldots,S_k)$ be a system and $(G_1,G_2)$ be a pair of subgroups of $\Z^k$.  Define the {\em space of $(G_1,G_2)$-cubes} as $$\bQ_{G_1,G_2}(X)\coloneqq \overline{\{(x,S_{g_1}x,S_{g_2}x,S_{g_1+g_2}x):  x\in X, g_1 \in G_1, g_2\in G_2\}}\subseteq X^{4}.$$
(Remark that such definitions were initially introduced in \cite{Donoso_Sun_cubes_product_ext:2015}, for the case where each $G_i$ is generated by a single transformation.)
Given $x\in X$, and $(G_1,G_2)$, let 
\[ \overline{\mathcal{F}_{G_1,G_2}}(x)\coloneqq \overline{\{(S_{g_1}x,S_{g_2}x,S_{g_1+g_2}x): g_1\in G_1, g_2\in G_2\}}\subseteq X^3.\]

When $G_i$ is generated by a single element $g_i$ we write $\bQ_{G_i,G_i}(X)$ simply as $\bQ_{S_{g_i},S_{g_i}}(X)$; a similar notation is used for  $\overline{F_{G_i,G_i}}(x)$. Given a single subgroup $G$ of $\Z^k$,  we write $$\RP_{G}(X)\coloneqq\overline{\{(x,S_{g}x):  x\in X, g \in G\}}\subseteq X^{2}$$  (this relation is coined the prolongation relation in \cite{Auslander_Guerin_regio_prox_and_prolongation:1997}). Note that if $(X,(S_m)_{m\in G})$ is transitive (meaning that there exists a point $x$ such that $\{S_mx : m \in G \}$ is dense in $X$), then $\RP_{G}(X)=X\times X$.
Similarly to \cite{Donoso_Sun_cubes_product_ext:2015} (or \cite{Host_Kra_Maass_nilstructure:2010} for the case of a $\Z$-action), we define the relation $\RP_{G_1,G_2}(X)$ as the set of points $(x,y)\in X^2$ such that $(x,y,y,y)\in \bQ_{G_1,G_2}(X)$. It should be noted that if $(X,G)$ is minimal, then $\RP_{G,G}(X)$ is nothing more than the classical regionally proximal relation (see \cite[Chapter 9]{Auslander_minimal_flows_and_extensions:1988} for more information on this relation). 

We need the following lemma.

\begin{lemma}\label{basicq}
     Let $(X,S_{1},\dots,S_{k})$ be a $\Z^k$-system and $G_{1},G_2$ be subgroups of $\Z^k$.
\begin{enumerate}[label=(\roman*)]
 \item[$(i)$] Let $\sigma\colon X^4\to X^4$ be the map with $\sigma(a,b,c,d)\coloneqq(a,c,b,d)$. Then $\sigma(\bQ_{G_1,G_2}(X))=\bQ_{G_2,G_1}(X)$.
 \item[$(ii)$] Consider the system $(\RP_{G_1}(X),G_2^{(2)})$, where $G_2^{(2)}$ is the action given by $g(x,y)=(gx,gy)$, for all $g\in G_2$ and $(x,y)\in X^2.$\footnote{Note that $\RP_{G_1}(X)$ is invariant under this action since $G_1$ and $G_2$ commute.} Then $\RP_{G_2^{(2)}}(\RP_{G_1}(X))=\bQ_{G_1,G_2}(X)$.
 \item[$(iii)$] If $H$ is a subgroup of $\Z^k$, and $\RP_{G_1}(X)=\RP_{G_2}(X)$, then $\bQ_{G_1,H}(X)=\bQ_{G_2,H}(X)$. 
\item[$(iv)$] If $G_1',G_2'$ and $G_1,G_2$ are subgroups of $\Z^k$ such that $G_1'\subseteq G_1$, and $G_2'\subseteq G_2$, then $\bQ_{G_1',G_2'}(X)$ $\subseteq \bQ_{G_1,G_2}(X).$
     \end{enumerate}
 \end{lemma}
\begin{proof}
 $(i)$ and $(iv)$ follow directly from the definitions.
 
To show $(ii)$, first note that for all $x \in X$, $g_1 \in G_1$, and $g_2\in G_2$, the point $(x,S_{g_1}x,S_{g_2}x,S_{g_1+g_2}x)$ belongs to $\RP_{G_2^{(2)}}(\RP_{G_1}(X))$ (here we naturally identify this point with $((x,S{g_1}x),(S_{g_2}x,S_{g_1+g_2}x))$). Therefore, $\bQ_{G_1,G_2}(X)\subseteq \RP_{G_2^{(2)}}(\RP_{G_1}(X))$. For the converse inclusion, it suffices to show that for any $(x,y)\in \RP_{G_1}(X)$ and any $g_2\in G_2$, we have $(x,y,S_{g_2}x,S_{g_2}y)\in \bQ_{G_1,G_2}(X)$. Let $\epsilon>0$ and choose $0<\delta<\epsilon$ so that if $z,z'\in X$ and $\rho(z,z')<\delta$, then $\rho(S_{g_2}z,S_{g_2}z')<\epsilon$. We can find $x'\in X$ and $g_1\in G_1$ such that $\rho((x',S_{g_1}x'),(x,y))<\delta$. It follows that $( x',S_{g_1}x',S_{g_2}x',S_{g_1+g_2}x')$ is at distance at most $\epsilon$ of $(x,y,S_{g_2}x,S_{g_2}y)$. Since $\epsilon>0$ is arbitrary, we conclude that $(x,y,S_{g_2}x,S_{g_2}y)\in \bQ_{G_1,G_2}(X)$, as desired.

 $(iii)$ follows immediately from $(ii)$. 
\end{proof}
 
\begin{corollary}  \label{cor:transitive_RP}
   Let $(X,S_1,\ldots,S_k)$ be a $\Z^k$-system and let $G\subseteq \Z^k$ be a subgroup such that $(X,(S_m)_{m\in G})$ is transitive. Then, for any subgroup $H$ of $\Z^k$, we have $\bQ_{G,H}(X)=\bQ_{\Z^k,H}(X)$. In particular, $\RP_{G,G}(X)=\RP_{\Z^k,\Z^k}(X)$.  
\end{corollary}

\begin{proof}
Note that by \cref{basicq} $(iv)$, the inclusion $\bQ_{G,H}(X)\subseteq \bQ_{\Z^k,\Z^k}(X)$ always holds. In addition, if $G$ is transitive, since $\RP_{G}(X)=X\times X=\RP_{\Z^k}(X)$, \cref{basicq} $(iii)$ implies that $\bQ_{G,H}(X)=\bQ_{\Z^k,H}(X)$. Using $(i)$ and $(iii)$ of \cref{basicq}, we get $\bQ_{G,G}(X)=\bQ_{\Z^k,\Z^k}(X),$ from where $\RP_{G,G}(X)=\RP_{\Z^k,\Z^k}(X)$.
\end{proof}

We remark that Corollary \ref{cor:transitive_RP} could also be deduced by using a proof similar to that of \cite[Lemma~6.13]{Donoso_Sun_cubes_product_ext:2015}.

The following theorem is classical (see, for instance, \cite{Auslander_minimal_flows_and_extensions:1988}). 

\begin{theorem} \label{thm:factor_rotation}
  Let $(X,S_1,\ldots,S_k)$ be a minimal system. Then $\RP_{\Z^k,\Z^k}(X)$ is an equivalence relation, the system $X/\RP_{\Z^k,\Z^k}$ is the maximal equicontinuous factor of $X$, and this factor is topologically conjugate to a rotation on a compact abelian group.  
\end{theorem}

\subsection{The $O$-diagram} 
The following is a classical theorem in the structural theory of topological dynamical systems and will be very useful for our purposes. We state a version for $\Z^k$, giving only the information we need for our work. We note that this theorem is valid for general group actions. For further details, the interested reader may consult \cite[Chapter VI, Section 3]{deVries_elements_topological_dynamics:1993} or \cite[Chapter 14]{Auslander_minimal_flows_and_extensions:1988}. 

\begin{theorem}
\label{thm:RIC} \label{thm:O-diagram}
Let $(X,S_1,\ldots,S_k)$ and $(Y, R_1,\ldots,R_k)$ be two $\Z^k$-minimal systems and $\pi\colon X \to Y$ a factor map. Then there exist two $\Z^k$-minimal systems $(\tilde{X},\tilde{S}_1,\ldots,\tilde{S}_k)$ and $(\tilde{Y}, \tilde{R}_1,\ldots,\tilde{R}_k)$ and factor maps $\tilde{\pi}\colon \tilde{X}\to X$, $\tilde{\sigma}\colon \tilde{X}\to X$, $\tilde{\tau}\colon \tilde{Y}\to Y$ 
such that the following diagram  (which is called the $O$-diagram)
\[
\begin{tikzcd}
	\tilde{X}\arrow[r, "\tilde{\sigma}"] \arrow[d,"\tilde{\pi}" ]
	&X\arrow[d, "\pi " ] \\
	\tilde{Y} \arrow[r,  "\tilde{\tau}" ]
	& Y \end{tikzcd}
\]
 is commutative, $\tilde{\sigma}$ and $\tilde{\tau}$ are almost one-to-one, and $\tilde{\pi}$ is open.
\end{theorem}

\cref{thm:RIC} says that, modulo almost one-to-one extensions, we may assume that the factor map is open.

\section{The proof of the main result}\label{S:3}
\subsection{A characterization for the regional proximal relation for product transformations}

The following is the main tool we use in the proof of \cref{mainbb} and can be interpreted as a topological analogue of seminorm control estimates in product spaces (see \cite[Lemma 5.2]{Donoso_Koutsogiannis_Sun_seminorms_polynomials_joint_ergodicity:2022} or \cite[Lemma 3.4]{Donoso_Ferre_Koutsogiannis_Sun_multicorr_joint_erg:2024} for analogous statements in the measurable setting).

\begin{theorem} \label{lemma:product_RP}
Let $(X,S_1,\dots,S_k)$ be a $\Z^k$-system and $T_1,\ldots,T_d \in \langle S_1,\ldots, S_k\rangle$. Let $(Y_i,S_1,\ldots,S_k)$, $1\leq i\leq d$ be factors of $X$, such that for all $i$, the factor map $\pi_i\colon X\to Y_i$ is open, and $R_{\pi_i}\subseteq \RP_{T_i,T_i}(X)$. For all $1\leq i\leq d$, let $(x_i,y_i)\in R_{\pi_i}$.
    Then 
    \[ ((x_1,\dots,x_d),(y_1,\dots,y_d))\in \RP_{T_1\times\dots\times T_d}(X^{d}).  \]  
\end{theorem}
Interpreting appropriately the coordinates, this can be stated as    ``$\RP_{T_1,T_1}(X)\times\dots\times \RP_{T_d,T_d}(X)\subseteq  \RP_{T_1\times\dots\times T_d}(X^{d})$.''

We need some lemmas before proving \cref{lemma:product_RP}. The following lemma can be deduced from the proof of \cite[Theorem 7.3.2]{Huang_Shao_Ye_nilbohr_automorphy:2016}, or using \cite{Veech_equicontinuous_structure_abelian_groups:1968}.
We give a proof for completeness.

\begin{lemma} \label{lem:Delta_set} Let $(X,T)$ be a $\Z$-system and let  $x,y\in X$ be such that $(x,y,x) \in {\overline{\mathcal{F}_{T,T}}}(x)$. Then for any open neighborhood $U$ of $y$, there is a sequence  $(a_i)_{i\in\N}\subseteq \Z$ of integers infinitely many values such that the set $\{n\in \Z : T^nx\in U\}$ contains $\{a_j-a_i:\;j>i\}$.
\end{lemma}

\begin{proof}
Let $\epsilon>0$ be such that $B(y,\epsilon)\subseteq U$. For $i\in \N$, set $\epsilon_i=\epsilon/2^i$. Construct a sequence $(\delta_i)_{i\in \N}$, with $0<\delta_i<\epsilon_i$ and  a sequence $(m_i,n_i)_{i\in \N}$ in $\Z\times \Z$   as follows:
Let $n_1,m_1$ be such that $\rho(T^{n_1}x,x)<\epsilon_1$, $\rho(T^{m_1}x,y)<\epsilon_1,$ and $\rho(T^{n_1+m_1}x,x)<\epsilon_1$. Pick $0<\delta_2<\epsilon_2$ such that $\rho(z,z')<\delta_2$ implies that $\rho(T^a z,T^a z')<\epsilon_2$ for all $|a|\leq |n_1|+|m_1|$. Take $n_2$, $m_2$ such that $\rho(T^{n_2}x,x)< \delta_2$, $\rho(T^{m_2}x,y)< \delta_2$ and $\rho(T^{n_2+m_2}x,x)<\delta_2$. (We highlight here that the numbers $n_2, m_2$ can be taken to be arbitrarily large.) Note that the definition of $\delta_2$ implies that:
\begin{multline*}
\rho(T^{n_2+n_1}x,x)< \epsilon_1+\epsilon_2, \ \rho(T^{n_2+n_1+m_1}x,x) < \epsilon_1+\epsilon_2, \ \rho(T^{n_2+m_2+n_1}x,x)<\epsilon_1+\epsilon_2,  \\ \rho(T^{n_2+m_2+n_1+m_1}x,x)< \epsilon_1+\epsilon_2, \ \rho(T^{n_2+m_1}x,y) < \epsilon_1+\epsilon_2, \;\; \text{and}\;\; \rho(T^{n_2+m_2+m_1}x,y) < \epsilon_1+\epsilon_2.
\end{multline*}
So, if we set $R_1=\{n_1,n_1+m_1\}$, $P_1=\{m_1\}$, $R_2=\{n_2,n_2+m_2\}$, we have that $\rho(T^{a+b}x,x)< \epsilon_1+\epsilon_2$ for all $a\in R_2$ and $b\in R_1$, and $\rho(T^{a+c}x,y)<\epsilon_1+\epsilon_2$ for all $a\in R_2$ and $c\in P_1$ (here $R$ stands for ``return'' and $P$ for ``passage''). 

The idea of the proof is that return times associated with large indices are compatible with return times and passages associated with smaller indices. More precisely, inductively, suppose that we have defined $\delta_i$, $m_i$ and $n_i$ for all $1\leq i\leq l$ for some $l\in\mathbb{N}$, and for the set $R_i=\{n_i,n_i+m_i\}$ and $P_i=\{m_i\}$ we have that if $a=r_{j_1}+\dots +r_{j_l}$, with $r_{j_k}\in R_{j_k}$, $j_1<\ldots < j_l$, then   $\rho(T^ax,x) <\sum_{t=1}^{l}  \epsilon_{j_t} $, and $\rho(T^{a+c}x,y)< \epsilon_k + \sum_{t=1}^{l} \epsilon_{j_t} $ if $c\in P_k$, $k<j_1$.  

Let $0<\delta_{i+1}<\epsilon_{i+1}$ be such that $\rho(z,z')<\delta_{i+1}$ implies that $\rho(T^a z,T^a z')<\epsilon_{i+1}$ for all $|a|\leq |n_1|+|m_1|+\cdots+ |n_i|+|m_i|$. Then choose $n_{i+1}$ and $m_{i+1}$ such that $\rho(T^{n_{i+1}}x,x)< \delta_{i+1}$, $\rho(T^{n_{i+1}+m_{i+1}}x,x)< \delta_{i+1},$ and $\rho(T^{m_{i+1}}x,y)< \delta_{i+1}$, and set $R_{i+1}=\{n_{i+1}, n_{i+1}+m_{i+1}\}$, $P_{i+1}=\{m_{i+1}\}$.

We claim that if $a=r_{j_1}+r_{j_2}+\dots + r_{j_l}$, with $r_{j_k}\in R_{j_k}$, $j_1<\ldots < j_l\leq i+1$, then $\rho(T^ax,x)< \sum_{t=1}^{l}  \epsilon_{j_t} $, and $\rho(T^{a+c}x,y)< \epsilon_k + \sum_{t=1}^{l} \epsilon_{j_t} $ if $c\in P_k$, $k<j_1$. 

We only need to check the case $j_l=i+1$. Assume that $r_{i+1}=n_{i+1}$ (the case $r_{i+1}=n_{i+1}+m_{i+1}$ is identical). Since $\rho(T^{n_{i+1}}x,x)< \delta_{i+1}$, and $|a-n_{i+1}|\leq |n_1|+|m_1|+\dots +|n_i|+|m_i|$, we get $\rho(T^{a}x,T^{a-n_{i+1}}x)< \epsilon_{i+1}$. By induction, $\rho(T^{a-n_{i+1}}x,x)< \sum_{t=1}^{l-1}\epsilon_{j_t}$, so the triangle inequality implies that $\rho(T^ax,x)<\sum_{t=1}^{l}\epsilon_{j_t}$, as desired. The estimate of $\rho(T^{a+c}x,y)$
follows in a similar way. 

Now set $a_i=n_{i+1}+\sum_{j=1}^i (n_j+m_j)$. Note that we may further require the sequence $(a_i)_{i\in\N}$ to take infinitely many values (by choosing the $n_i$'s and $m_i$'s to go to infinity) and $a_{i+l}-a_i =n_{i+l+1} + \sum_{j=i+2}^{i+l} (n_{j}+m_j) +m_{i+1}$. 
Hence, we can rewrite this as $a_{i+l}-a_i=r_{i+l+1}+\sum_{j=i+2}^{i+l} r_j + c$, where $r_j\in R_j$ and $c\in P_{i+1}$. It follows from the construction of the sequence $\{n_i,m_i:i \in \N\}$ that $\rho(T^{a_{i+l}-a_i}x,y)<\sum_{t=i+1}^{i+l+1}\epsilon_i< \epsilon$, which implies that $T^{a_{i+l}-a_i}x\in U,$ as was to be shown.
\end{proof}

\begin{lemma} \label{lem:recurrence_RP}
 Let $(X,S_1,\dots,S_k)$ be a minimal $\Z^k$-system, $T_1,\ldots,T_d\in \langle S_1,\ldots, S_k\rangle$ and $\pi\colon X\to Y$ an open factor map with $R_{\pi}\subseteq \RP_{T_i,T_i}(X)$ for some $1\leq i\leq d$. Let  $x_{1},\dots,x_{d},x'_{i}\in X$ with $(x_{i},x'_{i})\in R_{\pi}$, and $U_{z}$be  a neighborhood of $z$ for $z=x_{1},\dots,x_{d},x'_{i}$. There exist $n\in\mathbb{Z}$ such that $T_{i}^{n}U_{x_{i}}\cap U_{x'_{i}}\neq \emptyset$ and $T_{j}^{n}U_{x_{j}}\cap U_{x_{j}}\neq \emptyset$ for all $1\leq j\leq d, j\neq i$.   
\end{lemma}

\begin{proof}
The set $\Omega_i$ of $\tilde{x}\in X$ such that the set $\{(a,b,c):(\tilde{x},a,b,c)\in \bQ_{T_i,T_i}(X)
\}$ equals $\overline{\mathcal{F}_{T_{i},T_i}}(\tilde{x})$  is a dense $G_{\delta}$-set of points (see, for instance, \cite[Lemma 4.5]{Glasner_top_erg_decomposition:1994}). Since $\pi$ is open, we can find $\tilde{x}\in U_{x_i}\cap \Omega_{i}$ and $\tilde{y}\in U_{x_i'}$ with $(\tilde{x},\tilde{y})\in R_{\pi}$. Because $R_{\pi}\subseteq \RP_{T_i,T_i}(X)$, we have $(\tilde{x},\tilde{y},\tilde{x})\in \overline{\mathcal{F}_{T_{i},T_i}}(\tilde{x})$ and by \cref{lem:Delta_set}, the set $\{ n\in \Z: T_i^n \tilde{x} \in U_{x_i'}\}$ contains a set of the form $\{a_j-a_j':\;j>j'\}$ for some $\Z$-valued sequence $(a_i)_{i\in\N}$ taking infinitely many values. In particular, the same is true for the set $\{n\in \Z: T_i^n U_{x_i}\cap U_{x_i'}\neq \emptyset\}$. 
Let $\mu$ be a $\Z^k$-invariant measure on $X$. By the minimality of $\langle S_1, \ldots, S_k\rangle$,  $\mu$ has full support. Consider the product system $(X_1\times \cdots\times X_d,\mathcal{B}(X)^{\otimes d},\mu^{\otimes d},T_1\times \cdots\times  T_d)$ and $U=U_1\times\cdots\times U_d$. Then $\mu^{\otimes d}(U)>0,$ and so by the proof of the Poincar\'e recurrence theorem, the set $\{n\in\Z : \mu^{\otimes d}(U\cap (T_1\times \cdots\times T_d)^{-n}U)>0\}$ must intersect non-trivially every set of the form $\{a_j-a_j':\;j>j'\},$ which takes infinitely many values. This implies that it has non-empty intersection with $\{n\in \Z: T_i^n U_{x_i}\cap U_{x_i'}\neq \emptyset\}$, as desired.    
\end{proof}

\begin{proof}[Proof of \cref{lemma:product_RP}]
Fix $\epsilon>0$ and set $\epsilon_{d}\coloneqq\epsilon$. Suppose that we have constructed $\epsilon_{r+1},\dots,\epsilon_{d}>0$ for some $1\leq r\leq d-1$. We let $0<\epsilon_{r}<\epsilon_{r+1}/2$ to be a number such that for any $z_{r+1}\in X$ with $\rho(y_{r+1},z_{r+1})<\epsilon_{r}$, there exists $x'_{r+1}\in X$ with $\rho(x'_{r+1},x_{r+1})<\epsilon_{r+1}/2$ such that  $(x'_{r+1},z_{r+1})\in R_{\pi_{r+1}}$. The existence of such $\epsilon_{r}$ follows from the assumptions that $(x_{r+1},y_{r+1})\in R_{\pi_{r+1}}$,  the map $X\mapsto Y_{{r+1}}=X/R_{\pi_{r+1}}$ is open.

 For $1\leq r\leq d$,
we say that \emph{Property $r$}   holds if there exist
 $z_{1},\dots,z_{d}\in X$ and $n\in\mathbb{Z}$ such that
\begin{itemize}
\item $\rho(x_{i},z_{i})<\epsilon_{r}$ for all $1\leq i\leq r$;
\item $\rho(y_{i},z_{i})<\epsilon_{r}$ for all $r+1\leq i\leq d$;
\item $\rho(y_{i},T_{i}^{n}z_{i})<\epsilon_{r}$ for all $1\leq i\leq d$.
\end{itemize}

By \cref{lem:recurrence_RP}, Property 1 holds.
Now suppose that Property $r$ holds for some $1\leq r\leq d-1$.
Since $(x_{r+1},y_{r+1})\in R_{\pi_{r+1}}$ and $\rho(y_{r+1},z_{r+1})<\epsilon_{r}$, by the construction of $\epsilon_{r}$, there exists $x'_{r+1}\in X$ with $\rho(x'_{r+1},x_{r+1})<\epsilon_{r+1}/2$ such that  $(x'_{r+1},z_{r+1})\in R_{\pi_{r+1}}$.
Let $\delta'\coloneqq\min\{\epsilon_{r+1}/2, \epsilon_{r+1}-\epsilon_{r}\}$.
Take $0<\delta<\delta'$ such that for all $x,y\in X$, if $\rho(x,y)<\delta$, then $\rho(T^{n}x,T^{n}y)<\delta'$ ($n$ is the one from Property $r$ above).
By \cref{lem:recurrence_RP}, there exist 
$z'_{1},\dots,z'_{d}\in X$ and $n'\in\mathbb{Z}$ such that
\begin{itemize}
\item $\rho(z_{i},z'_{i})<\delta$ for all $1\leq i\leq d, i\neq r+1$;
\item $\rho(z_{i},T_{i}^{n'}z'_{i})<\delta$ for all $1\leq i\leq d, i\neq r+1$;
\item $\rho(x'_{r+1},z'_{r+1})<\delta$;
\item $\rho(z_{r+1},T_{r+1}^{n'}z'_{r+1})<\delta$.
\end{itemize} 

Then for all $1\leq i\leq r$, 
$$\rho(x_{i},z'_{i})\leq \rho(x_{i},z_{i})+\rho(z_{i},z'_{i})<\epsilon_{r}+\delta\leq\epsilon_{r+1}.$$
For all $r+2\leq i\leq d$, 
$$\rho(y_{i},z'_{i})\leq \rho(y_{i},z_{i})+\rho(z_{i},z'_{i})<\epsilon_{r}+\delta\leq\epsilon_{r+1}.$$
Moreover,
$$\rho(x_{r+1},z'_{r+1})\leq \rho(x_{r+1},x'_{r+1})+\rho(x'_{r+1},z'_{r+1})<\epsilon_{r+1}/2+\delta\leq\epsilon_{r+1}.$$
On the other hand, for all $1\leq i\leq d, i\neq r+1$, since $\rho(z_{i},T_{i}^{n'}z_{i})<\delta$, we have $\rho(T_{i}^{n}z_{i},T_{i}^{n+n'}z_{i})<\delta'$ and so
$$\rho(y_{i},T_{i}^{n+n'}z_{i})\leq \rho(y_{i},T_{i}^{n}z_{i})+\rho(T_{i}^{n}z_{i},T_{i}^{n+n'}z_{i})<\epsilon_{r}+\delta'\leq\epsilon_{r+1}.$$
Finally, since $\rho(z_{r+1},T_{r+1}^{n'}z_{r+1})<\delta$, we have that  $\rho(T_{r+1}^{n}z_{r+1},T_{r+1}^{n+n'}z_{r+1})<\delta'$ and so
$$\rho(y_{r+1},T_{r+1}^{n+n'}z_{r+1})\leq \rho(y_{r+1},T_{r+1}^{n}z_{r+1})+\rho(T_{r+1}^{n}z_{r+1},T_{r+1}^{n+n'}z_{r+1})<\epsilon_{r}+\delta'\leq\epsilon_{r+1}.$$
In conclusion, we have that  that Property $r+1$ holds. 

So it follows from induction that Property $d$ holds, which means that there exist $(z_{1},\dots,z_{d})\in X^{d}$ and $n\in\mathbb{Z}$ such that 
$$\rho((x_{1},\dots,x_{d}),(z_{1},\dots,z_{d}))< \epsilon \text{ and } \rho((y_{1},\dots,y_{d}),(T_{1}^{n}z_{1},\dots,T_{d}^{n}z_{d}))<\epsilon.$$
Since $\epsilon$ is arbitrary, we have that $((x_{1},\dots,x_{d}),(y_{1},\dots,y_{d}))\in\RP_{T_{1}\times\dots\times T_{d}}(X^{d})$.
\end{proof}

As a consequence of \cref{lemma:product_RP}, we have:

\begin{proposition}
\label{1sub}
    Let $(X,S_{1},\dots,S_{k})$ be a minimal $\mathbb{Z}^{k}$-system and $T_1,\ldots,T_d\in \langle S_1,\ldots,S_k\rangle$. Suppose that $(X,T_{1}),\dots,(X,T_{d})$ are transitive. Then $(X^{d},T_{1}\times \dots\times T_{d})$ is transitive if and only if $(Y^{d},T_{1}\times\dots\times T_{d})$ is transitive, where $Y=X/\bold{RP}_{\mathbb{Z}^{k},\mathbb{Z}^{k}}(X)$.
\end{proposition}

\begin{proof} 
 The ``only if'' part is straightforward. Now assume that $(Y^d,T_1\times\dots\times T_d)$ is transitive. By the O-diagram (\cref{thm:O-diagram}), we may consider almost one-to-one extensions $\tilde{X}$, $\tilde{Y}$  of $X$ and $Y$ respectively such that the projection $\tilde{\pi}\colon \tilde{X}\to \tilde{Y}$ is open. Note that $(\tilde{X},T_1),\dots,(\tilde{X},T_d)$ and $(\tilde{Y}^d,T_1\times\dots\times T_d)$ are also transitive, because this property is preserved under almost one-to-one extensions (see \cite{Akin_Glasner_residual_properties_almost_equi:2001}). 
 We now show that $(\tilde{X}^d,T_1\times\dots\times T_d)$ is transitive, which implies that $(X^d,T_1\times\dots\times T_d)$ is transitive.

Note that since $\tilde{X}$ is an almost one-to-one extension of $X$, we have that $\tilde{X}/\RP_{\Z^k,\Z^k}(\tilde{X})$ and $X/\RP_{\Z^k,\Z^k}(X)$ are conjugate, so we have that $R_{\tilde{\pi}}$ is a subset of $\RP_{T_i,T_i}(\tilde{X})$ for all $i$. To see this, by only assuming that $\tilde{\sigma}$ is proximal (which covers the almost one-to-one case), let $q$ be the projection from $\tilde{X}$ to $X/\RP_{\Z^k,\Z^k}(X)$. It suffices to show that $R_{q}\subseteq \RP_{\Z^k,\Z^k}(\tilde{X})$. Let $\tilde{x},\tilde{x}'$ with $q(\tilde{x})=q(\tilde{x}')$. Then $(\tilde{\sigma}(\tilde{x}),\tilde{\sigma}(\tilde{x}'))\in \RP_{\Z^k,\Z^k}(X)$. We can find $(\tilde{y}, \tilde{y}')\in \RP_{\Z^k,\Z^k}(\tilde{X})$ such that $ (\tilde{\sigma}(\tilde{y}),\tilde{\sigma}(\tilde{y}'))=(\tilde{\sigma}(\tilde{x}),\tilde{\sigma}(\tilde{x}'))$. It follows that $(\tilde{x},\tilde{y}), (\tilde{x}',\tilde{y}') \in P(\tilde{X})$ (the proximal relation on $\tilde{X}$). Since $P(\tilde{X})\subseteq \RP_{\Z^k,\Z^k}(\tilde{X})$, and this is an equivalence relation, we conclude $(\tilde{x},\tilde{x}')\in \RP_{\Z^k,\Z^k}(\tilde{X})$.\footnote{ This should be a well-known result; we chose to present its short proof (which simplifies the one of \cite[Lemma~5.3]{Donoso_Durand_Maass_Petite_automorphism_low_complexity:2016} that covers the $\mathbb{Z}$ case and almost one-to-one extensions) for completeness.} 

Let $U,V$ be nonempty open subsets of $\tilde{X}^{d}$. Then $\tilde{\pi}^{\times d}(U)$ and $\tilde{\pi}^{\times d}(V)$ are nonempty open sets, where $\tilde{\pi}^{\times d}\coloneqq \pi\times\cdots\times\pi$ ($d$-times). Since $(\tilde{Y}^d,S_1\times\dots\times S_d)$ is transitive, there exist $(x_{1},\dots,x_{d})\in U$ and $n\in \Z$ such that $\tilde{\pi}^{\times d}(T_1^nx_1,\ldots,T_d^nx_{d})\in \tilde{\pi}^{\times d}(V)$. That is, there exists $(x'_{1},\dots,x'_{d})\in V$ such that $(T_i^nx_{i},x'_{i})\in R_{\tilde{\pi}}(\tilde{X})$ for all $1\leq i\leq d$. Let $\epsilon>0$ be such that $B((x_{1},\dots,x_{d}),\epsilon)\subseteq U$ and $B((x'_{1},\dots,x'_{d}),\epsilon)\subseteq V$. Take $\delta>0$ so that $\rho(a,b)<\delta$ implies $\rho(T_i^{-n}a,T_i^{-n} b)<\epsilon$ for $1\leq i\leq d$. By \cref{lemma:product_RP}, as $R_{\tilde{\pi}}(\tilde{X})\subseteq \RP_{T_i,T_i}(\tilde{X})$, we obtain $((T_1^nx_{1},\dots,T_d^nx_{d}),(x'_{1},\dots,x'_{d}))\in \RP_{T_1\times\dots\times T_d}(\tilde{X}^d)$. Therefore, there exist $(y_{1},\dots,y_{d})\in \tilde{X}^d$ and $m\in \Z$ such that $\rho((y_1,\ldots,y_d),$ $(T_1^nx_{1},\dots,T_d^nx_{d}))<\delta$ and $\rho((T_1^my_{1},\dots,T_d^my_{d}),(x_1',\ldots,x_d'))<\delta$. It follows that $(T_1^{-n}y_{1},\dots,$ $T_d^{-n}y_{d})\in U$ and $(T_1^{m+n}T_1^{-n}y_{1},\dots,T_d^{m+n}T_d^{-n}y_{d})\in V$.  Using \cref{le:eqtr}, we conclude that $(\tilde{X}^{d},T_{1}\times\dots\times T_{d})$ is transitive.
\end{proof}

\subsection{The proof of Theorem \ref{mainbb}}

In this last subsection, we prove \cref{mainbb}.   We start with its forward direction, which is almost straightforward.

\begin{proof}[Proof of the forward direction of \cref{mainbb}] 
We use \cref{le:eqtr} implicitly throughout the proof.
 Assume that $(T_1^n)_n,\ldots,(T_d^n)_n$ are jointly transitive. Equivalently, for all $V_{0},\dots,V_{d}$ nonempty and open subsets of $X,$ there exists $n\in\Z$ such that 
\begin{equation}\label{E:1}
V_{0}\cap T^{-n}_{1}V_{1}\cap\dots\cap T^{-n}_{d}V_{d}\neq \emptyset.
\end{equation} 
To show $(i)$, pick any $i\neq j,$ and let $U_0, U_1$ be non-empty opens sets. Setting $V_i=U_0,$ $V_j=U_1$ and $V_k=X$ for all $k\in \{0,\ldots,d\}\setminus \{i,j\},$ it follows from \eqref{E:1} that $T_i^{-n}V_i\cap T_j^{-n}V_j\neq \emptyset$, or $U_0\cap (T_i^{-1}T_j)^{-n} U_1\neq \emptyset$.
To show $(ii)$, pick any point $x$ for which $\{(T_1^n x, \ldots, T_d^n x):\;n\in \mathbb{Z}\}$ is dense in $X.$ Then the point $(x,\ldots,x)$ is a transitive point for $T_1\times \dots\times T_d$.
\end{proof}

It remains to show the inverse direction of \cref{mainbb}. To this end, we first need a couple of statements. In particular, the first one will allow us to run an inductive argument.

\begin{proposition}\label{trfg}
 Let $(X,S_{1},\ldots,S_{k})$ be a minimal $\mathbb{Z}^{k}$-system, $T_1,\ldots,T_d\in \langle S_1,\ldots,S_k\rangle$, and $R_{2}\coloneqq T_{2}T_{1}^{-1}, \ldots, R_{d}\coloneqq T_{d}T_{1}^{-1}$. If $(X,R_{2}), \ldots, (X,R_{d})$ and $(X^{d},T_{1}\times\dots\times T_{d})$ are transitive, then $(X^{d-1},R_{2}\times\dots\times R_{d})$ is transitive.
\end{proposition}

\begin{proof}
Since $(X,R_{2}),\ldots, (X,R_{d})$ are transitive, by \cref{cor:transitive_RP}, we have $\RP_{R_{2},R_{2}}(X)=\ldots=\RP_{R_{d},R_{d}}(X)=\RP_{\Z^k,\Z^k}(X)$, which is an equivalence relation since $(X,S_{1},\ldots,S_{k})$ is minimal.  By \cref{1sub}, it suffices to show that $(Y^{d-1},R_{2}\times\dots\times R_{d})$ is transitive, where $Y=X/\bold{RP}_{\mathbb{Z}^{k},\mathbb{Z}^{k}}(X)$. By \cref{thm:factor_rotation} we have that $(Y,S_1,\ldots,S_k)$ is a rotation on a compact abelian group, so we may write $T_i(y)=y+\alpha_i$ for $\alpha_i\in Y$ for $1\leq i \leq d$. Since $(X^{d},T_{1}\times\dots\times T_{d})$ is transitive, we get that $(Y^d,T_1\times\dots\times T_d)$ is transitive, hence minimal. (This holds because rotations are distal, and in this class, transitivity and minimality are equivalent conditions--e.g., see \cite[Chapters 2 and 5]{Auslander_minimal_flows_and_extensions:1988}.)

Take $y\in Y$ and $(y_2,\ldots,y_d)\in Y^{d-1}$. Since $(Y^d,T_1\times \cdots \times T_d)$ is minimal, given $\epsilon>0$, there exists $n$ such that $\rho(y+n\alpha_1,y)<\epsilon$, and $\rho(y+n\alpha_i, y_i)<\epsilon$, for all $2\leq i \leq d$.  
It follows that $\rho(y+n(\alpha_i-\alpha_1) , y_i)<2\epsilon$ for all $2\leq i \leq d$, and since $y,y_2,\ldots,y_d$ are arbitrary, we get that $(Y^{d-1},R_2\times \cdots\times R_d)$ is minimal. \cref{1sub} allows us to conclude.  
\end{proof}

The following lemma is a generalization of \cite[Lemma 2.9]{Cao_Shao_top_mild_mixing_poly:2022} (see also \cite[Lemma 3]{Kwietniak_Oprocha_wm_minimality_disjointness:2012}).

\begin{lemma}\label{lemma:transive_shrinking}
Let $(X,S_{1},\dots,S_{k})$ be a $\mathbb{Z}^{k}$-system and $T_1,\ldots,T_d\in \langle S_1,\ldots,S_k\rangle$. Let $(R_{j})_{1\leq j\leq N}$ be a finite sequence of continuous maps from $X$ to $X$.
Assume that $(X^{d},T_{1}\times\dots\times T_{d})$ is transitive. Then for all non-empty open sets $V_{1},\dots,V_{d}$, there exists $n_{j}\in\Z, 1\leq j\leq N$, and for each $1\leq i\leq d$ a   non-empty open subset $\tilde{V}_{i}$ of $V_{i}$ such that 
$$T_{i}^{-n_{j}}R^{-1}_{j}\tilde{V}_{i}\subseteq V_{i} \text{ for all } 1\leq j\leq N, 1\leq i\leq d.$$
\end{lemma}

\begin{proof}
We use induction on $N$.
Since $(X^{d},T_{1}\times\dots\times T_{d})$ is transitive, there exists $n_{1}\in\Z$ such that 
$T_{i}^{-n_{1}}R^{-1}_{1}V_{i}\cap V_{i}\neq \emptyset$ for all $1\leq i\leq d$. Set $\tilde{V}_i=V^{(1)}_{i}\coloneqq T_{i}^{-n_{1}}R^{-1}_{1}V_{i}\cap V_{i}$. This completes the proof for the case $N=1$.

Now assume that for some $N\geq 2$, we have constructed $n_{1},\dots,n_{N-1}\in\mathbb{N}$ with $n_{1}<\ldots<n_{N-1}$, and for each $1\leq i\leq d$ a  sequence of non-empty open sets $V_{i}\supseteq V^{(1)}_{i}\supseteq \ldots\supseteq V^{(N-1)}_{i}$
such that 
$$T_{i}^{-n_{j}}R^{-1}_{j}V^{(m)}_{i}\subseteq V_{i} \text{ for all } 1\leq j\leq m, 1\leq m\leq N-1, 1\leq i\leq d.$$
Let $U_{i}:=R^{-1}_{N}V_{i}^{(N-1)}$. Since $(X^{d},T_{1}\times\dots\times T_{d})$ is transitive, there exists $n_{N}\in\mathbb{N}$ with $n_{N}>n_{N-1}$ such that
$T_{i}^{-n_{N}}U_{i}\cap V_{i}\neq \emptyset$ for all $1\leq i\leq d$.
This implies that
$$V_i\cap T_{i}^{-n_{N}}R^{-1}_{N}V_{i}^{(N-1)}=V_i\cap T_{i}^{-n_{N-1}}U_{i}\neq \emptyset.$$
Let
$$V_{i}^{(N)}:=V_{i}^{(N-1)}\cap (T_{i}^{-n_{N}}R^{-1}_{N})^{-1}V_{i}.$$
Then $V_{i}^{(N)}\subseteq V_{i}^{(N-1)}$ is an non-empty open set and
$T_{i}^{-n_{N}}R^{-1}_{N}V^{(N)}_{i}\subseteq V_{i}$. Since $V_{i}^{(N)}\subseteq V_{i}^{(N-1)}$, we also have that 
$$T_{i}^{-n_{j}}R^{-1}_{j}V^{(N)}_{i}\subseteq V_{i} \text{ for all } 1\leq j\leq N-1, 1\leq i\leq d.$$
This completes the induction step, and we are done by setting $\tilde{V}_{i}:=V_{i}^{(N)}$.
\end{proof}

\begin{proof}[Proof of the inverse direction of \cref{mainbb}]
There is nothing to prove when $d=1$. Now we assume that Theorem \ref{mainbb} holds for $d-1$ for some $d\geq 2$ and we prove it for $d$.
By \cref{trfg}, conditions (i) and (ii) imply that $(X^{d-1},T_2T_1^{-1}\times\dots\times T_{d}T_{1}^{-1})$ is transitive. On the other hand, we have that $(T_{i}T_{1}^{-1})^{-1}(T_{j}T_{1}^{-1})=T_{i}^{-1}T_{j}$ is transitive for all $1\leq i,j\leq d, i\neq j$. So, by induction hypothesis, we have that $((T_{2}T_{1}^{-1})^{n})_n,\ldots,((T_{d}T_{1}^{-1})^{n})_n$ are jointly transitive.

For $m=(m_1,\dots,m_k)\in \Z^k$, let $S_{m}B:=S_1^{m_1}\cdot\ldots\cdot S_k^{m_k}B$.  
  Let $U,V_{1},\dots,V_{d}$ be open and nonempty. 
We wish to show that there exists $n\in\Z$ such that
$$U\cap T_{1}^{-n}V_{1}\cap\dots\cap T^{-n}_{d}V_{d}\neq \emptyset.$$

Since $(X,S_{1},\dots,S_{k})$ is minimal,  there exists a finite set $F\subseteq \Z^k$ such that $X=\bigcup_{r\in F} S_{r}U$.  
By assumption (ii) and \cref{lemma:transive_shrinking}, there exist non-empty open sets $\tilde{V}_1,\dots,\tilde{V}_d$ and  for all $r\in F$ some $n_r\in \Z$, such that 
\[ (T_{1}\times\dots \times T_{d})^{n_{r}} (S_{k}\times\dots\times S_{k} (\tilde{V}_1\times\dots\times \tilde{V}_d))\subseteq V_1 \times\dots\times V_d\]
for all $r\in F$.
Since
$(T_{2}T_{1}^{-1})^{n},\dots,(T_{d}T_{1}^{-1})^{n}$ are jointly transitive, 
 we can find $m=m(F)\in \Z$ such that
\[ \tilde{V}_1 \cap (T_{2}T_1^{-1})^{-m}\tilde{V}_2\cap\dots\cap (T_{d}T_1^{-1})^{-m}\tilde{V}_d \neq \emptyset. \]
Take $x\in \tilde{V}_1$ such that $T_i^mT_1^{-m}x\in \tilde{V}_i$ for all $2\leq i\leq d$, and write $y=T_1^{-m}x$. Let $r\in F$ be such that $z:=S_{r}y\in U$. Then $T_{i}^{m+n_r}z=T_i^{m+n_r}S_{r}y=T_i^{n_r}S_{r}(T_{i}T_{1}^{-1})^{m}x\in T_i^{n_r}S_{r}(\tilde{V}_{i})\subseteq V_{i}$ for all $1\leq i\leq d$.
It follows that $z\in U\cap T_1^{-(m+n_r)}V_1\cap\dots\cap T_d^{-(m+n_r)}V_d$. 
\end{proof}

We close this article with the following widely open problem. Due to Theorem~\ref{mainbb} and recent developments in the theory of topological factors, we believe that there will be numerous results in the joint transitivity problem in the near future.

\begin{problem}\label{Pr:1}
Obtain the analogous to Theorem~\ref{mainbb} joint transitivity characterizations for iterates that come from polynomial and Hardy field of polynomial growth functions, for which the corresponding results in the measure theoretic setting are known.
\end{problem}

\medskip

\noindent {\bf{Acknowledgments.}} After sharing our findings with D. Charamaras, F. K. Richter, and K. Tsinas, we were informed that, by using elementary methods, they had obtained a special case of Theorem~\ref{mainbb}. Given the distinct nature of the methods, the two groups have decided not to publish their results together.


\begin{thebibliography}{10}
	
	\bibitem{Akin_Glasner_residual_properties_almost_equi:2001}
	E.~Akin and E.~Glasner.
	\newblock Residual properties and almost equicontinuity.
	\newblock {\em J. Anal. Math.}, 84:243--286, 2001.
	
	\bibitem{Auslander_minimal_flows_and_extensions:1988}
	J.~Auslander.
	\newblock {\em Minimal flows and their extensions}, volume 153 of {\em
		North-Holland Mathematics Studies}.
	\newblock North-Holland Publishing Co., Amsterdam, 1988.
	\newblock Notas de Matem\'{a}tica, 122. [Mathematical Notes].
	
	\bibitem{Auslander_Guerin_regio_prox_and_prolongation:1997}
	J.~Auslander and M.~Guerin.
	\newblock Regional proximality and the prolongation.
	\newblock {\em Forum Math.}, 9(6):761--774, 1997.
	
	\bibitem{Berend_Bergelson_joint_ergodicity:1984}
	D.~Berend and V.~Bergelson.
	\newblock Jointly ergodic measure-preserving transformations.
	\newblock {\em Israel J. Math.}, 49(4):307--314, 1984.
	
	\bibitem{Bergelson_WM_PET:1987}
	V.~Bergelson.
	\newblock Weakly mixing {PET}.
	\newblock {\em Ergodic Theory Dynam. Systems}, 7(3):337--349, 1987.
	
	\bibitem{Bergelson_Leibman_Son_joint_erg_generalized_linear:2016}
	V.~Bergelson, A.~Leibman, and Y.~Son.
	\newblock Joint ergodicity along generalized linear functions.
	\newblock {\em Ergodic Theory Dynam. Systems}, 36(7):2044--2075, 2016.
	
	\bibitem{Bergelson_Moreira_Richter_mult_averages_conv_recurrence_comb_app:2024}
	V.~Bergelson, J.~Moreira, and F.~K. Richter.
	\newblock Multiple ergodic averages along functions from a {H}ardy field:
	{C}onvergence, recurrence and combinatorial applications.
	\newblock {\em Adv. Math.}, 443:Paper No. 109597, 2024.
	
	\bibitem{Cao_Shao_top_mild_mixing_poly:2022}
	Y.~Cao and S.~Shao.
	\newblock Topological mild mixing of all orders along polynomials.
	\newblock {\em Discrete Contin. Dyn. Syst.}, 42(3):1163--1184, 2022.
	
	\bibitem{Chu_Frantzikinakis_Host_ergodic_averages_distinct_degree:2011}
	Q.~Chu, N.~Frantzikinakis, and B.~Host.
	\newblock Ergodic averages of commuting transformations with distinct degree
	polynomial iterates.
	\newblock {\em Proc. Lond. Math. Soc. (3)}, 102(5):801--842, 2011.
	
	\bibitem{deVries_elements_topological_dynamics:1993}
	J.~de~Vries.
	\newblock {\em Elements of topological dynamics}, volume 257 of {\em
		Mathematics and its Applications}.
	\newblock Kluwer Academic Publishers Group, Dordrecht, 1993.
	
	\bibitem{Donoso_Durand_Maass_Petite_automorphism_low_complexity:2016}
	S.~Donoso, F.~Durand, A.~Maass, and S.~Petite.
	\newblock On automorphism groups of low complexity subshifts.
	\newblock {\em Ergodic Theory Dynam. Systems}, 36(1):64--95, 2016.
	
	\bibitem{Donoso_Ferre_Koutsogiannis_Sun_multicorr_joint_erg:2024}
	S.~Donoso, A.~Ferr\'{e}~Moragues, A.~Koutsogiannis, and W.~Sun.
	\newblock Decomposition of multicorrelation sequences and joint ergodicity.
	\newblock {\em Ergodic Theory Dynam. Systems}, 44(2):432--480, 2024.
	
	\bibitem{Donoso_Koutsogiannis_Kuca_Tsinas_Sun_multiple_Hardy}
	S.~Donoso, A.~Koutsogiannis, B.~Kuca, K.~Tsinas, and W.~Sun.
	\newblock Joint ergodicity for commuting transformations along hardy field
	functions (temporary title).
	\newblock In preparation.
	
	\bibitem{Donoso_Koutsogiannis_Sun_pointwise_sublinear:2020}
	S.~Donoso, A.~Koutsogiannis, and W.~Sun.
	\newblock Pointwise multiple averages for sublinear functions.
	\newblock {\em Ergodic Theory Dynam. Systems}, 40(6):1594--1618, 2020.
	
	\bibitem{Donoso_Koutsogiannis_Sun_seminorms_polynomials_joint_ergodicity:2022}
	S.~Donoso, A.~Koutsogiannis, and W.~Sun.
	\newblock Seminorms for multiple averages along polynomials and applications to
	joint ergodicity.
	\newblock {\em J. Anal. Math.}, 146(1):1--64, 2022.
	
	\bibitem{Donoso_Koutsogiannis_Sun_joint_erg_poly_growth:2023}
	S.~Donoso, A.~Koutsogiannis, and W.~Sun.
	\newblock Joint ergodicity for functions of polynomial growth.
	\newblock 2023.
	\newblock To appear in Israel Journal of Mathematics.
	
	\bibitem{Donoso_Sun_cubes_product_ext:2015}
	S.~Donoso and W.~Sun.
	\newblock Dynamical cubes and a criteria for systems having product extensions.
	\newblock {\em J. Mod. Dyn.}, 9:365--405, 2015.
	
	\bibitem{Frantzikinakis_mult_recurrence_Hardy_poly_growth:2010}
	N.~Frantzikinakis.
	\newblock Multiple recurrence and convergence for {H}ardy sequences of
	polynomial growth.
	\newblock {\em J. Anal. Math.}, 112:79--135, 2010.
	
	\bibitem{Frantzikinakis_multidim_Szemeredi_Hardy:2015}
	N.~Frantzikinakis.
	\newblock A multidimensional {S}zemer\'{e}di theorem for {H}ardy sequences of
	different growth.
	\newblock {\em Trans. Amer. Math. Soc.}, 367(8):5653--5692, 2015.
	
	\bibitem{Frantzikinakis_joint_ergodicity_sequences:2023}
	N.~Frantzikinakis.
	\newblock Joint ergodicity of sequences.
	\newblock {\em Adv. Math.}, 417:Paper No. 108918, 63, 2023.
	
	\bibitem{Frantzikinakis_Kra_averages_product_integrals:2005}
	N.~Frantzikinakis and B.~Kra.
	\newblock Polynomial averages converge to the product of integrals.
	\newblock {\em Israel J. Math.}, 148:267--276, 2005.
	
	\bibitem{Frantzikinakis_Kuca_joint_erg_commuting:2022}
	N.~Frantzikinakis and B.~Kuca.
	\newblock Joint ergodicity for commuting transformations and applications to
	polynomial sequences.
	\newblock 2022.
	\newblock Preprint. arXiv:2207.12288.
	
	\bibitem{Frantzikinakis_Kuca_seminorm_control_commuting_poly:2022}
	N.~Frantzikinakis and B.~Kuca.
	\newblock Seminorm control for ergodic averages with commuting transformations
	and pairwise dependent polynomial iterates.
	\newblock 2022.
	\newblock To appear in Ergodic Theory Dynam. Systems.
	
	\bibitem{Furstenberg_ergodic_szemeredi:1977}
	H.~Furstenberg.
	\newblock Ergodic behavior of diagonal measures and a theorem of
	{S}zemer\'{e}di on arithmetic progressions.
	\newblock {\em J. Anal. Math.}, 31:204--256, 1977.
	
	\bibitem{Glasner_top_erg_decomposition:1994}
	E.~Glasner.
	\newblock Topological ergodic decompositions and applications to products of
	powers of a minimal transformation.
	\newblock {\em J. Anal. Math.}, 64:241--262, 1994.
	
	\bibitem{Host_Kra_Maass_nilstructure:2010}
	B.~Host, B.~Kra, and A.~Maass.
	\newblock Nilsequences and a structure theorem for topological dynamical
	systems.
	\newblock {\em Adv. Math.}, 224(1):103--129, 2010.
	
	\bibitem{Huang_Shao_Ye_nilbohr_automorphy:2016}
	W.~Huang, S.~Shao, and X.~Ye.
	\newblock Nil {B}ohr-sets and almost automorphy of higher order.
	\newblock {\em Mem. Amer. Math. Soc.}, 241(1143):v+83, 2016.
	
	\bibitem{Huang_Shao_Ye_top_correspondence_multiple_averages:2019}
	W.~Huang, S.~Shao, and X.~Ye.
	\newblock Topological correspondence of multiple ergodic averages of nilpotent
	group actions.
	\newblock {\em J. Anal. Math.}, 138(2):687--715, 2019.
	
	\bibitem{Karageorgos_Koutsogiannis_integer_indep_poly_averages_and_primes:2019}
	D.~Karageorgos and A.~Koutsogiannis.
	\newblock Integer part independent polynomial averages and applications along
	primes.
	\newblock {\em Studia Math.}, 249(3):233--257, 2019.
	
	\bibitem{Koutsogiannis_integer_poly_correlation:2018}
	A.~Koutsogiannis.
	\newblock Integer part polynomial correlation sequences.
	\newblock {\em Ergodic Theory Dynam. Systems}, 38(4):1525--1542, 2018.
	
	\bibitem{Koutsogiannis_multiple_averages_variable_poly:2022}
	A.~Koutsogiannis.
	\newblock Multiple ergodic averages for variable polynomials.
	\newblock {\em Discrete Contin. Dyn. Syst.}, 42(9):4637--4668, 2022.
	
	\bibitem{Koutsogiannis_Sun_total_joint_ergodicity:2023}
	A.~Koutsogiannis and W.~Sun.
	\newblock Total joint ergodicity for totally ergodic systems.
	\newblock 2023.
	\newblock Preprint. arXiv:2302.12278.
	
	\bibitem{Kwietniak_Oprocha_wm_minimality_disjointness:2012}
	D.~Kwietniak and P.~Oprocha.
	\newblock On weak mixing, minimality and weak disjointness of all iterates.
	\newblock {\em Ergodic Theory Dynam. Systems}, 32(5):1661--1672, 2012.
	
	\bibitem{Lehrer_mixing_unique_erg_models:1987}
	E.~Lehrer.
	\newblock Topological mixing and uniquely ergodic systems.
	\newblock {\em Israel J. Math.}, 57(2):239--255, 1987.
	
	\bibitem{Moothathu_diagonal_points:2010}
	T.~K.~S. Moothathu.
	\newblock Diagonal points having dense orbit.
	\newblock {\em Colloq. Math.}, 120(1):127--138, 2010.
	
	\bibitem{Qiu_poly_orbits_tot_minimal:2023}
	J.~Qiu.
	\newblock Polynomial orbits in totally minimal systems.
	\newblock {\em Adv. Math.}, 432:Paper No. 109260, 34, 2023.
	
	\bibitem{Shao_Ye_regionally_prox_orderd:2012}
	S.~Shao and X.~Ye.
	\newblock Regionally proximal relation of order {$d$} is an equivalence one for
	minimal systems and a combinatorial consequence.
	\newblock {\em Adv. Math.}, 231(3-4):1786--1817, 2012.
	
	\bibitem{Tsinas_joint_erg_Hardy:2023}
	K.~Tsinas.
	\newblock Joint ergodicity of {H}ardy field sequences.
	\newblock {\em Trans. Amer. Math. Soc.}, 376(5):3191--3263, 2023.
	
	\bibitem{Veech_equicontinuous_structure_abelian_groups:1968}
	W.~A. Veech.
	\newblock The equicontinuous structure relation for minimal {A}belian
	transformation groups.
	\newblock {\em Amer. J. Math.}, 90:723--732, 1968.
	
	\bibitem{Zhang_Zhao_topological_mult_rec_WM_GP:2021}
	R.~F. Zhang and J.~J. Zhao.
	\newblock Topological multiple recurrence of weakly mixing minimal systems for
	generalized polynomials.
	\newblock {\em Acta Math. Sin. (Engl. Ser.)}, 37(12):1847--1874, 2021.
	
\end{thebibliography}
\end{document}